\documentclass{amsart}
\usepackage[final]
{showkeys}

\usepackage{color}

\usepackage{hyperref}

\theoremstyle{theorem}
\newtheorem{theorem}{Theorem}[section]
\newtheorem{proposition}[theorem]{Proposition}
\newtheorem{lemma}[theorem]{Lemma}
\newtheorem{corollary}[theorem]{Corollary}

\theoremstyle{definition}

\newtheorem{remark}[theorem]{Remark}

\theoremstyle{definition} 
\newtheorem{example}[theorem]{Example}

\newcommand{\R}{\mathbb{R}}
\newcommand{\Z}{\mathbb{Z}}

\newcommand{\perm}{\operatorname{perm}}

\newcommand{\dd}{\operatorname{d}}
\renewcommand{\dd}{\,\mathrm{d}}

\newcommand{\one}{\mathbf{1}}

\newcommand{\psif}{\psi\bullet f}

\newcommand{\esssup}{\operatorname{ess\,sup}}
\newcommand{\essinf}{\operatorname{ess\,inf}}

\newcommand{\card}{\operatorname{card}}

\newcommand{\intr}[2]{\overline{#1,#2}}

\renewcommand{\l}{\mathbin{\Join
}}

\newcommand{\col}{\colon\!}

\newcommand{\al}{\alpha}
\newcommand{\be}{\beta}
\newcommand{\ga}{\gamma}
\newcommand{\de}{\delta}
\newcommand{\vp}{\varepsilon}
\newcommand{\la}{\lambda}
\newcommand{\La}{\Lambda}
\newcommand{\vpi}{\varphi}
\newcommand{\si}{\sigma}
\newcommand{\Si}{\Sigma}

\newcommand{\tPsi}{\tilde\Psi}

\newcommand{\F}{\mathcal{F}}

\renewcommand{\L}{\mathcal{L}}
\newcommand{\RR}{\mathcal{R}}

\newcommand{\sm}{\setminus}
\newcommand{\we}{\wedge}

\newcommand{\m}{\overline{m}}

\numberwithin{equation}{section} 

\begin{document}

\title[
Generalized semimodularity: order statistics]{
Generalized semimodularity: order statistics
}
\markright{
$2$-semimodularity implies the $n$-semimodularity}
\author{Iosif Pinelis}

\address{Department of Mathematical Sciences, 
Michigan Technological University, 
Houghton, MI 49931 USA}
\email{ipinelis@mtu.edu}



%
%
%
%

%


%

%


\date{\today}


\maketitle

\begin{abstract}
A notion of generalized $n$-semimodularity is introduced, which extends that of (sub/super)mod\-ularity in four ways at once. The main result of this paper, stating that every generalized $(n\col2)$-semimodular function on the $n$th Cartesian power of a distributive lattice is generalized $n$-semi\-modular, may be considered a multi/infinite-dimensional analogue of the well-known Muirhead lemma in the theory of Schur majorization. 
This result is also similar to a discretized version of the well-known theorem due to Lorentz, which latter was given only for additive-type functions. 
Illustrations of our main result are presented for counts of combinations of faces of a polytope; one-sided potentials; multiadditive forms, including multilinear ones -- in particular, permanents of rectangular matrices and elementary symmetric functions; and association inequalities for order statistics. Based on an extension of the FKG inequality due to Rinott \& Saks and Aharoni \& Keich,  
applications to correlation inequalities for order statistics are given as well. 
%
%
\end{abstract}

\tableofcontents


\section{Summary and discussion}\label{results} 

As pointed out e.g.\ in \cite{bach13,bach15}, the notion of submodularity has become useful in various areas: combinatorial optimization, with many applications in operations research; 
machine learning; computer vision; electrical networks; signal processing; several areas of theoretical
computer science, such as matroid theory; economics. One may also note the use of this notion in potential theory \cite{choquet}, as a capacity is a submodular function. 

%
%

Let $L$ be any distributive lattice; for definitions and facts pertaining to lattices, see e.g.\ \cite{graetzer11}. 


A function $\la\colon L\to\R$ is called submodular if 
\begin{equation}\label{eq:submod}
	\la(f)+\la(g)\ge\la(f\vee g)+\la(f\wedge g)
\end{equation}
for all $f$ and $g$ in $L$. A function $\la$ is called supermodular if the function $-\la$ is submodular, and $\la$ is called modular if it is both submodular and supermodular. See\ e.g.\ \cite{topkis,milgrom-roberts,fujishige,narayanan,topkis98,bach15}. 
Let us say that a function $\mu$ log-submodular if $\ln\mu$ is submodular. 
The log-submodularity condition and the corresponding log-supermodularity condition were referred to in Karlin and Rinott \cite{karlin-rinott80-I,karlin-rinott80-II} as the multivariate total positivity of order 2 (MTP${}_2$) and the multivariate reverse rule of order 2 (MRR${}_2$), respectively. 
As noted by Choquet \cite[\S 14.3]{choquet}, a nondecreasing function $\la$ is alternating of order $2$ iff it satisfies inequality \eqref{eq:submod}, that is, $\la$ is submodular; it was also shown in \cite{choquet} that the classical Newtonian capacity is such a function. 

The log-supermodularity condition is the condition under which the famous Fortuin--Kasteleyn--Ginibre (FKG) correlation inequality \cite{fkg} holds. Therefore, using inequality \eqref{eq:prod} 
together with the FKG inequality and its generalizations, we will be able to obtain the corresponding applications, in Corollaries~\ref{cor:fkg} and \ref{cor:AhKe}.


More generally, let $\RR$ be any set, endowed with a transitive relation $\l$, so that for any $a,b,c$ in $\RR$ one has the implication $a\l b\ \&\ b\l c\implies a\l c$. 
For any natural $n$, let us say that a function $\La\colon L^n\to\RR$ is \emph{generalized $n$-semimodular} if 
\begin{equation*}
	\La(f_1,\dots,f_n)\l\La(f_{n:1},\dots,f_{n:n})
\end{equation*}
for all $f=(f_1,\dots,f_n)\in L^n$, where $f_{n:1},\dots,f_{n:n}$ are the ``order statistics'' for $f$ defined by the formula 
\begin{equation}\label{eq:wedge-vee}
	f_{n:j}=\bigwedge\Big\{\bigvee_{i\in J}f_i\colon J\in\binom{[n]}j\Big\}  
\end{equation}
for $j\in[n]:=\intr1n$, with $\binom{[n]}j$ denoting the set of all subsets $J$ of the set $[n]
$ such that the cardinality of $J$ is $j$. 
Here and in the sequel we use the notation $\intr\al\be:=\{j\in\Z\colon\al\le j\le\be\}$. 
In particular, $f_{n:1}=f_1\wedge\dots\wedge f_n$ and $f_{n:n}=f_1\vee\dots\vee f_n$.

For any function $\la\colon L\to\R$, let the function $\La_\la\colon L^2\to\R$ be given by the formula $\La_\la(f,g):=\la(f)+\la(g)$ for $f$ and $g$ in $L$. 
Then, obviously, $\la$ is submodular or supermodular or modular if and only if $\La_\la$ is generalized $2$-semimodular with the relation ``$\l$'' being ``$\ge$'' or ``$\le$'' or ``$=$'', respectively.   

Thus, the notion of generalized $n$-semimodularity extends that of (sub/super)mod\-ularity in four ways at once: 
(i) the function $\La$ may be a function of any natural number $n$ of arguments, whereas $\la$ is a function of only one argument; (ii) in contrast with a general form of dependence of $\La(f_1,\dots,f_n)$ on $f_1,\dots,f_n$, the function $\La_\la$ of two arguments is of the special form, linear in $\la(f)$ and $\la(g)$; (iii) whereas the values of $\la$ are real numbers, those of $\La$ may be in any set $\RR$; and (iv) we now have an arbitrary transitive relation $\l$ over $\RR$ instead of one of the three particular relations ``$\ge$'' or ``$\le$'' or ``$=$'' over $\R$. 



For any $k\in[n]$, let us say that a function $\La\colon L^n\to\R$ is \emph{generalized $(n\col k)$-semimodular} if 
for each $j\in\intr0{n-k}$ and each $(n-k)$-tuple $(f_i\colon i\in[n]\,\setminus\,\intr{j+1}{j+k})\break 
\in L^{n-k}$ the function $L^k\ni(f_{j+1},\dots,f_{j+k})\mapsto\La(f_1,\dots,f_n)$ is generalized $k$-semimodular. 
In particular, $\La$ is generalized $(n\col n)$-semimodular if and only if it is generalized $n$-semimodular. 

Whenever the relation ``$\l$'' is denoted as ``$\ge$'' or ``$\le$'' or ``$=$'', let us replace ``semi'' in the above definitions by ``sub'', ``super'', and ``'', respectively. For instance, 
``generalized $n$-modular'' will stand for ``generalized $n$-semimodular'' with the relation ``$\l$'' being ``$=$''. 


The main result of this note is 

\begin{theorem}\label{th:}
Again, let $L$ be any distributive lattice. 
If a function $\La\colon L^n\to\RR$ is generalized $(n\col2)$-semimodular, then it is generalized  $n$-semi\-modular.  
\end{theorem}

The necessary proofs will be given in Section~\ref{proof}. 

As will be seen from the proof of Theorem~\ref{th:}, the condition that the function $\La$ be generalized $(n\col2)$-semimodular can be relaxed to the following: for 
each $j\in\intr1{n-1}$ and each $f=(f_1,\dots,f_n)\in L^n$ \emph{such that $f_1\le\dots\le f_j$}, one has 
$L(f_1,\dots,f_n)\l L(f_1,\dots,f_{j-1},f_j\wedge f_{j+1},f_j\vee f_{j+1},f_{j+2},\dots,f_n)$. 


\begin{remark}\label{rem:counterex}
Theorem~\ref{th:} will not hold in general if the lattice $L$ is not assumed to be distributive. For instance, let 
$L$ be defined by the set $[5]=\{1,2,3,4,5\}$ with the partial order being the subset of the natural order $\le$ on the set $[5]$ with elements $2,3,4$ now considered non-comparable with one another, so that the resulting order relation is the set $
\{(f,f)\colon f\in[5]\}\cup\{(1,2),(1,3),(1,4),(2,5),(3,5),(4,5),\break 
(1,5)\}$; then, in particular, $2\wedge3=1$ and $2\vee3=5$. This lattice is one of the simplest examples of non-distributive lattices. It is 
isomorphic to the diamond lattice $\mathsf{M}_3$ -- see e.g.\ \cite[page~110]{graetzer11}. Let $n=3$, $\RR=\R$, and  define the function $\La\colon L^3\to\R$ by the formula $\La(f_1,f_2,f_3):=12 f_1f_2 + 3 f_2f_3 + 5 f_1f_3$ for all $f=(f_1,f_2,f_3)\in L^3$. 
Then one can verify directly -- by a straightforward but tedious calculation consisting in checking $2\times5^3=250$ inequalities, two inequalities for each $f=(f_1,f_2,f_3)\in[5]^3$ --  
that this function $\La$ is generalized $(3\col2)$-submodular. 
However, $\La$ is not generalized $3$-
sub\-modular, because for $f=(2,3,4)$ one has $(f_{3:1},f_{3:2},f_{3:3})=(1,5,5)$ and $\La(f_1,f_2,f_3)=\La(2,3,4)=148\not\ge 160=\La(1,5,5)=\La(f_{3:1},f_{3:2},f_{3:3})$. \qed
\end{remark}

\begin{remark}\label{rem:pointwise}
A well-known fact, which will be crucial in the proof of Theorem~\ref{th:}, is the 
representation theorem due to Birkhoff and Stone stating  
that any distributive lattice $L$ is isomorphic to a lattice of subsets of (and hence to a lattice of nonnegative real-valued functions on) a certain set $S$, depending on $L$ (see e.g.\ \cite[Theorem 119]{graetzer11}). 
For such a lattice of functions, the ``order statistics'' $f_{n:1},\dots,f_{n:n}$ are uniquely determinined by the condition that 
\begin{equation}\label{eq:ord}
	f_{n:1}(s)\le\dots\le f_{n:n}(s)\quad\text{and}\quad \{\{f_{n:1}(s),\dots,f_{n:n}(s)\}\}=\{\{f_1(s),\dots,f_n(s)\}\}  
\end{equation}
for each $s\in S$, 
where the double braces are used to denote multisets, with appropriate multiplicities. 
To quickly see why this is true, one may reason as follows: Let us now use condition \eqref{eq:ord} to \emph{define} $f_{n:1},\dots,f_{n:n}$. Note that the value of 
the right-hand side (rhs) of \eqref{eq:wedge-vee} at any point $s\in S$ 
is invariant with respect to all permutations of the values $f_1(s),\dots,f_n(s)$. So, the 
value of the rhs of \eqref{eq:wedge-vee} at $s$ 
will not change if one replaces there $f_1,\dots,f_n$ by $f_{n:1},\dots,f_{n:n}$, and this value will equal $f_{n:j}(s)$. 
Thus, 
the definition of $f_{n:1},\dots,f_{n:n}$ by means of formula \eqref{eq:ord} is equivalent to the one given by \eqref{eq:wedge-vee}, if the lattice $L$ is already a lattice of real-valued functions on $S$. 
Moreover, it is clear now that, if the lattice $L$ is distributive, then definition \eqref{eq:wedge-vee} can be rewritten in the dual form, as 
\begin{equation}\label{eq:vee-wedge}
	f_{n:j}=\bigvee\Big\{\bigwedge_{i\in J}f_i\colon J\in\binom{[n]}{n+1-j}\Big\}  
\end{equation} 
for all $j\in[n]$. 

On the other hand, it can be seen that, if $L$ is not distributive, then this duality can be lost and each of the definitions \eqref{eq:wedge-vee} and \eqref{eq:vee-wedge} of $f_{n:j}$ can be rather unnaturally skewed up or down. For instance, in the counterexample given in Remark~\ref{rem:counterex}, for $f=(2,3,4)$ we had $(f_{3:1},f_{3:2},f_{3:3})=(1,5,5)$ according to definition \eqref{eq:wedge-vee}, but we would have $(f_{3:1},f_{3:2},f_{3:3})=(1,1,5)$ according to \eqref{eq:vee-wedge}. 
 
However, one may note that the right-hand side of \eqref{eq:vee-wedge} is always $\le$ than that of \eqref{eq:wedge-vee}; this follows because for any $J\in\binom{[n]}{n+1-j}$ and any $K\in\binom{[n]}j$ there is some $k\in J\cap K$, and then $\bigwedge_{i\in J}f_i\le f_k\le\bigvee_{i\in K}f_i$. \qed
\end{remark}


In view of the lattice representation theorem cited in Remark~\ref{rem:pointwise}, 
Theorem~\ref{th:} may be considered a multi/infinite-dimensional analogue of  
the well-known Muirhead lemma in the theory of Schur majorization (cf.\ e.g.\ \cite[Lemma 2.B.1, page~32]{marsh-ol09}),  
which may be stated as follows: for vectors $x$ and $y$ in $\R^n$ such that $x\prec y$ (that is, $x$ is majorized by $y$), there exist finitely many vectors $x_0,\dots,x_m$ in $\R^n$ such that $x=x_0\prec\cdots\prec x_m=y$. 
However, no direct multi-dimensional extension of the Muirhead
lemma seems to exist, even in two dimensions (see e.g.\ \cite[page~11]{left_publ}). 

For functions that are ``infinite-dimensional'' counterparts of the ``$m$-dimensional'' function $\La\colon L^m\to\R$ given by the formula of the additive form 
\begin{equation}\label{eq:lorentz}
	\La(g_1,\dots,g_m)=\sum_{j=1}^m\la_j(g_j), 
\end{equation} 
Lorentz \cite{lorentz53} obtained a result similar to Theorem~\ref{th:};  
for readers' convenience, let us reproduce it here: For each $j\in[n]$, let 
$f_j^*$  denote the equimeasurable decreasing rearrangement \cite{HLP} of a 
function $f_j\colon(0,1)\to\R$. Let a real-valued expression $\Phi(x,u_1,\dots,u_n)$ be continuous in $(x,u_1,\dots,u_n)\in(0,1)\times[0,\infty)\times\cdots\times[0,\infty)$. Then the inequality 
\begin{equation}
	\int_0^1\Phi(x,f_1(x),\dots,f_n(x))\,dx\le\int_0^1\Phi(x,f_1^*(x),\dots,f_n^*(x))\,dx
\end{equation}
holds 
for all bounded positive measurable 
functions $f_1,\dots,f_n$ from $(0,1)$ to $\R$ if and only if the following two conditions hold: 
\begin{equation}\label{eq:Phi1}
	\Phi(u_i+h,u_j+h)-\Phi(u_i+h,u_j)-\Phi(u_i,u_j+h)+\Phi(u_i,u_j)\ge0
\end{equation}
and 
\begin{equation}\label{eq:Phi2}
	\int_0^\delta\big[\Phi(x-t,u_i+h)-\Phi(x-t,u_i)
	-\Phi(x+t,u_i+h)+\Phi(x+t,u_i)\big]\,dt\ge0
\end{equation}
for all $h>0$, $x\in(0,1)$, $\delta\in(0,x\wedge(1-x))$, $(u_1,\dots,u_n)\in[0,\infty)^n$, and $i,j$ in $[n]$ such that $i<j$; here, in each of inequalities \eqref{eq:Phi1} and \eqref{eq:Phi2}, the arguments of $\Phi$ that are the same for all the four instances of $\Phi$ are omitted, for brevity. 
 
To establish the connection between Lorentz's result and our Theorem~\ref{th:}, suppose e.g.\ that each of the functions $f_1,\dots,f_n$ in \cite{lorentz53} is a step function, constant on each of the intervals $(\frac{j-1}m,\frac jm]$ for $j\in[m]$, and then let $g_j(s):=f_s(\frac jm)$ for $j\in[m]$ and $s\in S:=[n]$. In fact, in the proof in \cite{lorentz53} the result is first established for such step functions $f_1,\dots,f_n$. It is also shown in \cite{lorentz53} that, for such ``infinite-dimensional'' counterparts of the functions given by the ``additive'' formula \eqref{eq:lorentz}, the sufficient condition is also necessary. 
In turn, as pointed out in \cite{lorentz53}, the result there generalizes an inequality in \cite{ruderman}.  
Another proof of a special case of the result in \cite{lorentz53} was given in 
\cite{borell73_pacific}.


\section{Illustrations and applications}\label{examples}

\subsection{A general construction of generalized $n$-submodular functions from submodular ones}\label{schur}

Recall here some basics of majorization theory \cite{marsh-ol09}. For $x=(x_1,\dots,x_n)$ and $y=(y_1,\dots,y_n)$ in $\R^n$,  write $x\prec y$ if $x_1+\dots+x_n=y_1+\dots+y_n$ and $x_{n\col1}+\dots+x_{n\col k}\ge y_{n\col 1}+\dots+y_{n\col k}$ for all $k\in[n]$. For any $D\subseteq\R^n$, a function $F\colon D\to\R$ is called Schur-concave if for any $x$ and $y$ in $D$ such that $x\prec y$ one has $F(x)\ge F(y)$. If $D=I^n$ for some open interval $I\subseteq\R$ and the function $F$ is continuously differentiable then, by Schur's theorem \cite[Theorem~A.4]{marsh-ol09}, $F$ is Schur-concave iff $(\frac{\partial F}{\partial x_i}-\frac{\partial F}{\partial x_j})(x_i-x_j)\le0$ for all $x=(x_1,\dots,x_n)\in D$. 


\begin{proposition}\label{prop:schur}
Suppose that a real-valued function $\la$ defined on a distributive lattice $L$ is submodular and nondecreasing, and a function $\R^n\ni x=(x_1,\dots,x_n)\to F(x_1,\dots,x_n)$ is nondecreasing in each of its $n$ arguments and Schur-concave. Then the function $\La=\La_{\la,F}\colon L^n\to\R$ defined by the formula 
\begin{equation}\label{eq:La schur}
	\La(f_1,\dots,f_n):=\La_{\la,F}(f_1,\dots,f_n):=F(\la(f_1),\dots,\la(f_n))
\end{equation}
for $(f_1,\dots,f_n)\in L^n$ 
is generalized $(n\col2)$-submodular and hence generalized $n$-submodular. 
\end{proposition} 

A rather general construction of submodular functions on rings of sets is provided by \cite[\S 23.2]{choquet}, which implies that $\cup$-homomorpisms preserve the property of being alternating of a given order, and the proposition at the end of \cite[\S 23.1]{choquet}, which describes general $\cup$-homomorpisms as maps of the form 
\begin{equation*}
	S\supseteq A\mapsto G(A):=\{t\in T\colon 
	(s,t)\in G\text{ for some }s\in A\}, 
\end{equation*}
where $S$ and $T$ are sets and $G\subseteq S\times T$; in the case when $G$ is (the graph of) a map, the above notation $G(A)$ is of course consistent with that for the image of a set $A$ under the map $G$; according to the definition in the beginning of \cite[\S 23]{choquet}, a $\cup$-homomorpism is a map $\vpi$ of set rings defined by the condition $\vpi(A\cup B)=\vpi(A)\cup\vpi(B)$ for all relevant sets $A$ and $B$. 

Therefore and because an additive function on a ring of sets is modular and hence submodular, we conclude that functions of the form 
\begin{equation}\label{eq:mu()}
	A\mapsto
	\mu(G(A)) 
\end{equation}
are submodular, 
where $\mu$ is a measure or, more generally, an additive function (say on a discrete set, to avoid matters of measurability). From this observation, one can immediately obtain any number of corollaries of Proposition~\ref{prop:schur} such as the following: 

\begin{corollary}\label{cor:schur}
Let $P$ be a polytope of dimension $d$. For each $\al\in\intr0d$, let $\F_\al$ denote the set of all $\al$-faces (that is, faces of dimension $\al$) of $P$. For any distinct $\al,\be,\ga$ in $\intr0d$, let $G=G_{\al,\be,\ga}$ be the set of all pairs $\big(f_\al,(f_\be,f_\ga)\big)\in\F_\al\times(\F_\be\times\F_\ga)$ such that 
$f_\al\cap f_\be\ne\emptyset$, $f_\al\cap f_\ga\ne\emptyset$, and $f_\be\cap f_\ga\ne\emptyset$. 
Let $L$ be a lattice of subsets of $\F_\al$. Let a function $\R^n\ni x=(x_1,\dots,x_n)\to F(x_1,\dots,x_n)$ be nondecreasing in each of its $n$ arguments and Schur-concave. Then 
the function $\La=\La_{\al,\be,\ga}\colon L^n\to\R$ defined by the formula 
\begin{equation*}
	\La(A_1,\dots,A_n):=F(\card G(A_1),\dots,\card G(A_n))
\end{equation*}
for $(A_1,\dots,A_n)\in L^n$ 
is generalized $(n\col2)$-submodular and hence generalized $n$-submodular. 
\end{corollary} 

For readers' convenience, here is a direct verification of the fact that maps of the form \eqref{eq:mu()} are submodular: noting that $G(A\cup B)=G(A)\cup G(B)$ and $G(A\cap B)\subseteq G(A)\cap G(B)$ and using the additivity of $\mu$, we have 
\begin{equation*}
	\mu(G(A\cup B))+\mu(G(A\cap B))
	\le \mu(G(A)\cup G(B))+\mu(G(A)\cap G(B))=\mu(G(A))+\mu(G(B)) 
\end{equation*}
for all relevant sets $A$ and $B$. 


\subsection{Generalized one-sided potential}\label{potential}
Let here $L$ be the lattice of all measurable real-valued functions on a measure space $(S,\Si,\mu)$, 
with the pointwise lattice operations $\vee$ and $\wedge$. Consider 
the 
the function $\La\colon L^n\to\RR$ given by the formula 
\begin{equation}\label{eq:La}
	\La(f_1,\dots,f_n):=\La_{\vpi,\psi}(f_1,\dots,f_n):=\sum_{j,k=1}^n\Psi(f_j-f_k)
\end{equation} 
for all $f=(f_1,\dots,f_n)\in L^n$, where 
\begin{equation}\label{eq:Psi}
\Psi(g):=\psi\Big(\int_S(\vpi\circ g )\dd\mu\Big)	
\end{equation}
for all $g\in L$, 
$\vpi\colon\R\to[0,\infty]$ is a nondecreasing or nonincreasing function, and 
$\psi\colon[0,\infty]\to(-\infty,\infty]$ is a 
concave function. 
Thus, the function $\La=\La_{\vpi,\psi}$ 
may be referred to as a generalized one-sided potential, since the function $\vpi$ is assumed to be monotonic. 

\begin{proposition}\label{prop:potential}
The function $\La=\La_{\vpi,\psi}$ defined by formula \eqref{eq:La} is generalized $(n\col2)$-submodular and hence generalized $n$-submodular. 
\end{proposition}


\subsection{Symmetric sums of nonnegative 
multiadditive functions}\label{multiadd}

Let $k$ be a natural number. 
Let $L$ be a sublattice of the lattice $\R^S$ of all real-valued functions on a set $S$. Let us say that the lattice $L$ is \emph{complementable} if $f\sm g:=f-f\wedge g\in L$ for any $f$ and $g$ in $L$, so   that $f=f\we g+f\sm g$. 
Assuming that $L$ is complementable, 
let us say that a function $m\colon L\to\R$ is \emph{additive} if 
\begin{equation*}
	m(f)=m(f\we g)+m(f\sm g)
\end{equation*}
for all $f$ and $g$ in $L$;  
further, 
let us say that a function $m\colon L^k\to\R$ is \emph{multiadditive} or, more specifically, \emph{$k$-additive} if $m$ is additive in each of its $k$ arguments, that is, if for each $j\in[k]$ and each $(k-1)$-tuple $(f_i\colon i\in[k]\sm\{j\})$ the function $L\ni f_j\mapsto m(f_1,\dots,f_k)$ is additive.

To state the main result of this subsection, we shall need the following notation: for any set $J$, let $\Pi_k^J$ denote the set of all $k$-permutations of $J$, that is, the set of all injective maps of the set $[k]$ to $J$.  

\begin{proposition}\label{prop:multiadd}
Suppose that $k$ and $n$ are natural numbers such that $k\le n$, $L$ is a complementable sublattice of $\R^S$, and $m\colon L^k\to\R$ is a nonnegative multiadditive function. Then 
the function $\La_m\colon L^n\to\R$ defined by the formula 
\begin{equation}\label{eq:La_m}
	\La_m(f_1,\dots,f_n):=\sum_{\pi\in\Pi_k^{[n]}}m(f_{\pi(1)},\dots,f_{\pi(k)}) 
\end{equation}
for $(f_1,\dots,f_n)\in L^n$ is generalized $(n\col 2)$-submodular and hence generalized $n$-submodular.  
\end{proposition}


Formula \eqref{eq:La_m} can be rewritten in the following symmetrized form: 
\begin{equation}\label{eq:La_m,rewr}
	\La_m(f_1,\dots,f_n)=k!\,\sum_{I\in\binom{[n]}k}\m(f_I),  
\end{equation}
where, for $I=\{i_1,\dots,i_k\}$ with $1\le i_1<\dots<i_k\le n$, 
\begin{equation}\label{eq:m bar}
	\m(f_I):=\m(f_{i_1},\dots,f_{i_k}):=\frac1{k!}\,\sum_{\pi\in\Pi_k^I}m(f_{\pi(1)},\dots,f_{\pi(k)});    
\end{equation}
note that the so-defined function $\m\colon L^k\to\R$ is multiadditive and nonnegative, given that $m$ is so. Also, $\m$ is permutation-symmetric in the sense that \break  $\m(f_{\pi(1)},\dots,f_{\pi(k)})=\m(f_1,\dots,f_k)$ for all $(f_1,\dots,f_k)\in L^k$ and all permutations $\pi\in\Pi_k^{[k]}$. 

\begin{example}\label{ex:multilin}
If $V$ is a vector sublattice of the lattice $\R^S$ and $L$ is the lattice of all nonnegative functions in $V$ 
then, clearly, $L$ is complementable and the restriction to $L^k$ of  
any multilinear function from $V^k$ to $\R$ is multiadditive. 

In particular, if $\mu$ is a measure on a $\si$-algebra $\Si$ over $S$, $V$ is a vector sublattice of $L^k(S,\Si,\mu)$, and $L$ is the lattice of all nonnegative functions in $V$, then the function $m\colon L^k\to\R$ given by the formula 
$$m(f_1,\dots,f_k):=\int_S f_1\cdots f_k \dd\mu$$ 
for $(f_1,\dots,f_k)\in L^k$ 
is 
multiadditive. 

So, by Proposition~\ref{prop:multiadd}, the functions $\La_m$ corresponding to the functions $m$ presented above in this example are generalized $(n\col 2)$-submodular and hence generalized $n$-submodular.  

Let now $B=(b_{i,j})$ be a $d\times p$ matrix with $d\le p$ and nonnegative entries $b_{i,j}$. The  permanent of $B$ is defined by the formula 
\begin{equation*}
	\perm B:=\sum_{J\in\binom{[p]}d}\perm B_{\cdot J},
\end{equation*}
where $B_{\cdot J}$ the square submatrix of $B$ consisting of the columns of $B$ with column indices in the set $J\in\binom{[p]}d$; and for a square $d\times d$ matrix $B=(b_{i,j})$, 
\begin{equation*}
	\perm B :=\sum_{\pi\in\Pi_d^{[d]}}b_{1,\pi(1)}\cdots b_{d,\pi(d)}. 
\end{equation*}
So, $\perm B$ is a multilinear function of the $d$-tuple $(b_{1,\cdot},\dots,b_{d,\cdot})$ of the rows of $B$. Also, if $d=p$, then $\perm B$ is a multilinear function of the $d$-tuple $(b_{\cdot,1},\dots,b_{\cdot,d})$ of the columns of $B$.
If $d\ge p$, then $\perm B$ may be defined by the requirement that the permanent be invariant with respect to transposition. 

Thus, from Proposition~\ref{prop:multiadd} we immediately obtain 

\begin{corollary}\label{cor:perm}
Assuming that the entries $b_{i,j}$ of the $d\times p$ matrix $B$ 
are nonnegative, 
$\perm B$ is a generalized $d$-submodular function of the $d$-tuple $(b_{1,\,\cdot},\dots,b_{d,\,\cdot})$ of its rows and a generalized $p$-submodular function of the $p$-tuple $(b_{\cdot,\,1},\dots,b_{\cdot,\,p})$ of its columns (with respect to the standard lattice structures on $\R^{1\times p}$ and $\R^{d\times1}$, respectively): 
\begin{align*}
	\perm
	\begin{pmatrix}
	b_{d\col 1,\,\cdot}\\ \vdots \\ b_{d\col d,\,\cdot}
\end{pmatrix}
&\le
	\perm
	\begin{pmatrix}
	b_{1,\,\cdot}\\ \vdots \\ b_{d,\,\cdot}
\end{pmatrix}[=\perm B],  
\\ 
\perm(b_{\cdot,\,p\col 1},\dots,b_{\cdot,\,p\col})&\le\perm(b_{\cdot,\,1},\dots,b_{\cdot,\,p})[=\perm B]. 
\end{align*}
\end{corollary}
Note that the condition $d\le p$ is not needed or assumed in Corollary~\ref{cor:perm}.

Yet another way in which multilinear and hence multiadditive functions may arise is via 
the elementary symmetric polynomials. 
Let $n$ be any natural number, and let $k\in[n]
$. 
The elementary symmetric polynomials are defined by the formula  
\begin{equation*}
	e_k(x_1,\dots,x_n):=\sum_{J\in\binom{[n]}k}\prod_{j\in J}x_j. 
\end{equation*}
In particular, $e_1(x_1,\dots,x_n):=\sum_{j\in[n]}x_j$ and $e_n(x_1,\dots,x_n):=\prod_{j\in[n]}x_j$. 

Let $f=(f_1,\dots,f_n)$ be the vector of 
measurable functions $f_1,\dots,f_n$ defined  
on a measure space $(S,\Si,\mu)$ with values in the interval $[0,\infty)$. 
Then it is not hard to see that the ``order statistics'' are nonnegative measurable functions as well. As usual, let $\mu(h):=\int_S h\,\dd\mu$. 

If the measure $\mu$ is a probability measure, then the functions $f_1,\dots,f_n$ are called random variables (r.v.'s) and, in this case, $f_{n:1},\dots,f_{n:n}$ will indeed be what is commonly referred to as the order statistics based on the ``random sample'' $f=(f_1,\dots,f_n)$; cf.\ e.g.\ \cite
{david-nagaraja}. In contrast with settings common in statistics, in general we do not impose any conditions 
on the joint or individual distributions of the r.v.'s $f_1,\dots,f_n$ -- except that these r.v.'s be nonnegative. 

Then we have the following. 
\begin{corollary}\label{th:symm poly}
\begin{equation}\label{eq:symm poly}
	e_k\big(\mu(f_1),\dots,\mu(f_n)\big)\ge e_k\big(\mu(f_{n:1}),\dots,\mu(f_{n:n})\big). 
\end{equation}
In particular, 
\begin{equation}\label{eq:prod}
	\mu(f_1)\cdots\mu(f_n)\ge\mu(f_{n:1})\cdots\mu(f_{n:n}). 
\end{equation}
\end{corollary}
This follows immediately from Proposition~\ref{prop:multiadd} and formula \eqref{eq:La_m,rewr}, since the product $\mu(f_1)\cdots\mu(f_k)$ is clearly multilinear and hence multiadditive in $(f_1,\dots,f_k)$. 
\end{example}

To deal with cases when some of the $\mu(f_j)$'s (or the $\mu(f_{n:j})$'s) equal $0$ and other ones equal $\infty$, let us assume here the convention $0\cdot\infty:=0$. 
One may note that, if the nonnegative functions $f_1,\dots,f_n$ are scalar multiples of one another or, more generally, if $f_{\pi(1)}\le\cdots\le f_{\pi(n)}$ for some permutation $\pi$ of the set $[n]$, then inequality \eqref{eq:symm poly} turns into the equality. 


As mentioned above, in Corollary~\ref{th:symm poly} it is not assumed that $f_1,\dots,f_n$ are independent r.v.'s. However, if $\mu$ is a probability measure and the r.v.'s $f_1,\dots,f_n$ are independent (but not necessarily identically distributed), then 
$\mu(f_1)\cdots\mu(f_n)=\mu(f_1\cdots f_n)=\mu(f_{n:1}\cdots f_{n:n})$ by the second part of \eqref{eq:ord}, and so, \eqref{eq:prod} can then be rewritten as the following positive-association-type inequality for the order statistics:  
\begin{equation}\label{eq:indep}
	\mu(f_{n:1}\cdots f_{n:n})\ge\mu(f_{n:1})\cdots\mu(f_{n:n}). 
\end{equation}

Let now $\psi$ be any monotone (that is, either nondecreasing or nonincreasing) function from $[0,\infty]$ to $[0,\infty]$. 
For $f=(f_1,\dots,f_n)$ as before, 
let 
$$\psif:=(\psi\circ f_1,\dots,\psi\circ f_n).$$ 
Then for $j\in[n]$ one has $(\psif)_{n:j}=\psi\circ f_{n:j}$ if $\psi$ is nondecreasing and $(\psif)_{n:j}=\psi\circ f_{n:n+1-j}$ if $\psi$ is nonincreasing. 
Thus, we have the following ostensibly more general forms of \eqref{eq:prod} and \eqref{eq:indep}:

\begin{corollary}\label{cor:psi}
\begin{equation}\label{eq:psi}
	\mu(\psi\circ f_1)\cdots\mu(\psi\circ f_n)\ge\mu\big((\psif)_{n:1})\cdots\mu((\psif)_{n:n}\big). 
\end{equation} 
If $\mu$ is a probability measure and the r.v.'s $f_1,\dots,f_n$ are independent, then  
\begin{equation}\label{eq:psi,indep}
	\mu\big((\psif)_{n:1}\cdots (\psif)_{n:n}\big)
	\ge\mu\big((\psif)_{n:1}\big)\cdots\mu\big((\psif)_{n:n}\big). 
\end{equation}
\end{corollary}


The property of the order statistics $f_{n:1},\cdots,f_{n:n}$ given by inequality \eqref{eq:psi,indep} may be called the diagonal positive orthant dependence -- cf.\ e.g.\ Definition~2.3 in \cite{joag-dev--proschan} of the negative orthant dependence.


Immediately from Theorem~\ref{th:} or from inequality \eqref{eq:psi} in Corollary~\ref{cor:psi}, one obtains 

\begin{corollary}\label{cor:1}
Take any $p\in\R\setminus\{0\}$. Then 
\begin{equation}\label{eq:p,r>0}
	\mu(f_1^p)^r\cdots\mu(f_n^p)^r\ge\mu(f_{n:1}^p)^r\cdots\mu(f_{n:n}^p)^r
\end{equation}
for any $r\in(0,\infty)$, and 
\begin{equation}\label{eq:p,r<0}
	\mu(f_1^p)^r\cdots\mu(f_n^p)^r\le\mu(f_{n:1}^p)^r\cdots\mu(f_{n:n}^p)^r
\end{equation}
for any $r\in(-\infty,0)$. 
Here we use the conventions 
$0^t:=\infty$ and $\infty^t:=0$ for $t\in(-\infty,0)$. We also the following conventions: $0\cdot\infty:=0$ concerning \eqref{eq:p,r>0} and $0\cdot\infty:=\infty$ concerning \eqref{eq:p,r<0}. 
\end{corollary}

Consider now the special case of Corollary~\ref{cor:1} with $r=1/p$. Letting then  
$p\to\infty$, we see that \eqref{eq:p,r>0} will hold with the $\mu(f_j^p)^r$'s and $\mu(f_{n:j}^p)^r$'s replaced there by $\mu\text{-}\!\esssup f_j$ and $\mu\text{-}\!\esssup f_{n:j}$, respectively, where $\mu\text{-}\!\esssup$ denotes the essential supremum with respect to measure $\mu$. This follows because $\mu(h^p)^{1/p}\underset{p\to\infty}\longrightarrow\mu\text{-}\!\esssup h$. 
Similarly, letting $p\to-\infty$, we see that \eqref{eq:p,r<0} will hold with the $\mu(f_j^p)^r$'s and $\mu(f_{n:j}^p)^r$'s replaced there by $\mu\text{-}\!\essinf f_j$ and $\mu\text{-}\!\essinf f_{n:j}$, respectively, where $\mu\text{-}\!\essinf$ denotes the essential infimum with respect to $\mu$. 
Moreover, considering (say) the counting measures $\mu$ on finite subsets of the set $S$ and noting that $\sup h=\sup_S h$ coincides with the limit of the net $(\max_J h)$ over the filter of all finite subsets $J$ of $S$, we conclude that \eqref{eq:p,r>0} will hold with the $\mu(f_j^p)^r$'s and $\mu(f_{n:j}^p)^r$'s replaced there by $\sup f_j$ and $\sup f_{n:j}$, respectively. \big(The statement about the limit can be spelled out as follows: 
$\sup_S h\ge \max_J h$ for all finite $J\subseteq S$, and for each real $c$ such that $c<\sup h$ there is some finite set $J_c\subseteq S$ such that for all finite sets $J$ such that $J_c\subseteq J\subseteq S$ one has $\max_J h>c$.\big) Similarly, \eqref{eq:p,r<0} will hold with the $\mu(f_j^p)^r$'s and $\mu(f_{n:j}^p)^r$'s replaced there by $\inf f_j$ and $\inf f_{n:j}$, respectively. Thus, we have 

\begin{corollary}\label{cor:max,min}
\begin{equation}\label{eq:sup}
	(\sup f_1)\cdots(\sup f_n)\ge(\sup f_{n:1})\cdots(\sup f_{n:n})
\end{equation}
and 
\begin{equation}\label{eq:inf}
	(\inf f_1)\cdots(\inf f_n)\le(\inf f_{n:1})\cdots(\inf f_{n:n}). 
\end{equation}
Here we use the following conventions: $0\cdot\infty:=0$ concerning \eqref{eq:sup} and $0\cdot\infty:=\infty$ concerning \eqref{eq:inf}. 
\end{corollary}

Alternatively, one can obtain \eqref{eq:sup} and \eqref{eq:inf} directly from Theorem~\ref{th:}. 

Also, of course there is no need to assume in Corollary~\ref{cor:max,min} that the functions $f_1,\dots,f_n$ are measurable. 

The special cases of inequalities \eqref{eq:p,r<0} and \eqref{eq:inf} for $n=2$ mean that the functions $h\mapsto\mu(h^p)^r$ and $h\mapsto\inf h$ are log-supermodular functions on the distributive lattice (say $\L_\Si$) of all nonnegative $\Si$-measurable functions on $S$ and on the distributive lattice (say $\L$) of all nonnegative functions on $S$, respectively. 

At this point, let us recall the famous Fortuin--Kasteleyn--Ginibre (FKG) correlation 
inequality \cite{fkg}, which states that for any log-supermodular function $\nu$ on a finite distributive lattice $L$ and any nondecreasing functions $F$ and $G$ on $L$ we have 
\begin{equation*}
	\nu(FG)\nu(1)\ge\nu(F)\nu(G),
\end{equation*}
where $\nu(F):=\sum_{f\in L}\nu(f)$. 

Then we immediately obtain 

\begin{corollary}\label{cor:fkg}
Let $\L^\circ_\Si$ be any finite sub-lattice of the lattice $\L_\Si$, and let $F$ and $G$ be nondecreasing functions from $\L^\circ_\Si$ to $\R$. Then 
\begin{equation*}
	\Big(\sum_{h\in\L^\circ_\Si}F(h)G(h)\mu(h)^r\Big)\Big(\sum_{h\in\L^\circ_\Si}\mu(h)^r\Big)
	\ge\Big(\sum_{h\in\L^\circ_\Si}F(h)\mu(h)^r\Big)\Big(\sum_{h\in\L^\circ_\Si}G(h)\mu(h)^r\Big)
\end{equation*}
for any $r\in(-\infty,0)$. 
Similarly, let $\L^\circ$ be any finite sub-lattice of the lattice $\L$, and let $F$ and $G$ be nondecreasing functions from $\L^\circ$ to $\R$. Then 
\begin{equation*}
	\Big(\sum_{h\in\L^\circ_\Si}F(h)G(h)\inf h\Big)\Big(\sum_{h\in\L^\circ_\Si}\inf h\Big)
	\ge\Big(\sum_{h\in\L^\circ_\Si}F(h)\inf h\Big)\Big(\sum_{h\in\L^\circ_\Si}G(h)\inf h\Big). 
\end{equation*}
\end{corollary} 

As shown by Ahlswede and Daykin \cite[pages 288--289]{ahls-daykin79}, their inequality \cite[Theorem~1]{ahls-daykin79} almost immediately implies, and is in a sense sharper than, the FKG inequality. Furthermore, Rinott and Saks \cite{rinott-saks93,rinott-saks92} and Aharoni and Keich \cite{aharoni-keich} independently obtained a more general inequality ``for $n$-tuples of nonnegative functions on a distributive lattice, of which the Ahlswede--Daykin inequality is the case $n = 2$.''
More specifically, in notation closer to that used in the present paper, \cite[Theorem~1.1]{aharoni-keich} states the following: 

Let $\al_1,\dots,\al_n,\be_1,\dots,\be_n$ be nonnegative functions defined on a distributive lattice $L$ such that 
\begin{equation*}
	\prod_{j=1}^n\al_j(f_j)\le\prod_{j=1}^n\be_j(f_{n:j})
\end{equation*}
for all $f_1,\dots,f_n$ in $L$. Then for any finite subsets $F_1,\dots,F_n$ of $L$ 
\begin{equation*}
	\prod_{j=1}^n\sum_{f_j\in F_j}\al_j(f_j)\le\prod_{j=1}^n\sum_{g_j\in F_{n:j}}\be_j(g_j), 
\end{equation*}
where 
\begin{equation*}
	F_{n:j}:=\{f_{n:j}\colon f=(f_1,\dots,f_n)\in F_1\times\dots\times F_n\}. 
\end{equation*}
Note that the definition of the ``order statistics'' used in \cite{aharoni-keich} is different from \eqref{eq:wedge-vee} in that their ``order statistics'' go in the descending, rather than ascending, order; 
also, the term ``order statistics'' is not used in \cite{aharoni-keich}. 

In view of this result of \cite{aharoni-keich} and our Corollaries~\ref{cor:1} and \ref{cor:max,min}, one immediately obtains the following statement, which generalizes and strengthens Corollary~\ref{cor:fkg}: 

\begin{corollary}\label{cor:AhKe}
Let $\F_1,\dots,\F_n$ be any finite subsets of the lattice $\L_\Si$. For each $j\in[n]$, let 
\begin{equation*}
	\F_{n:j}:=\{f_{n:j}\colon f=(f_1,\dots,f_n)\in\F_1\times\dots\times\F_n\}. 
\end{equation*}
Then 
\begin{equation}\label{eq:AhKe,r}
	\prod_{j=1}^n\sum_{f_j\in\F_j}\mu(f_j)^r\le\prod_{j=1}^n\sum_{h_j\in\F_{n:j}}\mu(h_j)^r
\end{equation}
for any $r\in(-\infty,0)$. 

Similarly, let now $\F_1,\dots,\F_n$ be any finite subsets of the lattice $\L$. Then 
\begin{equation*}
	\prod_{j=1}^n\sum_{f_j\in\F_j}\inf f_j\le\prod_{j=1}^n\sum_{h_j\in\F_{n:j}}\inf h_j. 
\end{equation*}
\end{corollary} 



Comparing inequalities \eqref{eq:p,r>0} and \eqref{eq:p,r<0} in Corollary~\ref{cor:1} or inequalities \eqref{eq:sup} and \eqref{eq:inf} in Corollary~\ref{cor:max,min}, 
one may wonder whether 
the FKG-type inequalities stated in Corollaries~\ref{cor:fkg} and \ref{cor:AhKe} for the functions $h\mapsto\mu(h)^r$ with $r<0$ and $h\mapsto\inf h$ admit of the corresponding reverse analogues for the functions $h\mapsto\mu(h)^r$ with $r>0$ and $h\mapsto\sup h$. 
However, it is not hard to see that such FKG-type inequalities are not reversible in this sense, a reason being that the sets $\F_{n:j}$ may be much larger than the sets $\F_j$. 

E.g., suppose that $n=2$, $S=\R$, $\mu$ is a Borel probability measure on $\R$, $0<\vp<\de<1$, $N$ is a natural number, $\F_1$ is the set of $N$ pairwise distinct constant functions $f_1,\dots,f_N$ on $\R$ such that $1-\vp<f_j<1+\vp$ for all $j\in[n]$, 
and $\F_2=\{g_1,\dots,g_N\}$, where $g_j:=(1-\de)\one_{(-\infty,j]}+(1+\de)\one_{(j,\infty)}$ and $\one_A$ denotes the indicator of a set $A$. 
Then it is easy to see that each of the sets $\F_{2:1}$ and $\F_{2:2}$ is of cardinality $N^2$. So, letting $\de\downarrow0$ (so that $\vp\downarrow0$ as well), we see that, for any real $r$, the right-hand side of \eqref{eq:AhKe,r} goes to $N^4$ whereas its left-hand side goes to $N^2$, which is much less than $N^4$ if $N$ is large.


\begin{example}\label{ex:multiadd-sets}
Closely related to Example \ref{ex:multilin} is as follows. Suppose that $(S,\Si)$ is a measurable space, $\mu$ is a measure on the product $\si$-algebra $\Si^{\otimes k}$, and $L$ is a subring of $\Si$. Then $L$ is complementable and 
%
%
the function $m\colon L^k\to\R$ given by the formula 
\begin{equation}\label{eq:mu-of-prod}
	m(A_1,\dots,A_k):=\mu(A_1\times\dots\times A_k) 
\end{equation}
for $(A_1,\dots,A_k)\in L^k$ is multiadditive. 

A particular case of formula \eqref{eq:mu-of-prod} is 
\begin{equation}\label{eq:card-of-prod}
	m(A_1,\dots,A_k):=\card\big(G\cap(A_1\times\dots\times A_k)\big),  
\end{equation}
where $\card$ stands for the cardinality and $G$ is an arbitrary subset of $S^k$. 
If $G$ is symmetric in the sense that $(s_1,\dots,s_k)\in G$ iff $(s_{\pi(1)},\dots,s_{\pi(k)})\in G$ for all permutations $\pi$ of the set $[k]$, then $G$ represents the set (say $E$) of all hyperedges of a $k$-uniform hypergraph over $S$, in the sense that $(s_1,\dots,s_k)\in G$ iff $\{s_1,\dots,s_k\}\in E$. 
\end{example}

We now have another immediate corollary of Proposition~\ref{prop:multiadd}: 

\begin{corollary}\label{cor:sets}
Suppose that $k$ and $n$ are natural numbers such that $k\le n$, $(S,\Si)$ is a measurable space, $\mu$ is a measure on the product $\si$-algebra $\Si^{\otimes k}$, and $L$ is a subring of $\Si$. Then
\begin{equation}\label{eq:sets}
	\sum_{\pi\in\Pi_k^{[n]}}\mu(A_{n\col\pi(1)}\times\dots\times A_{n\col\pi(k)})
	\le \sum_{\pi\in\Pi_k^{[n]}}\mu(A_{\pi(1)}\times\dots\times A_{\pi(k)}) 
\end{equation}
for all $(A_1,\dots,A_n)\in L^n$. 
\end{corollary}

\section{Proofs
}\label{proof} 

\bigskip

One may note that formula \eqref{eq:g_k} in the proof of Theorem~\ref{th:} below defines a step similar to a step in the process of the so-called insertion search (cf.\ e.g.\ \cite[Section~5.2.1]{knuth-vol3} (also called the sifting or sinking technique) -- except that here we do the pointwise comparison of functions (rather than numbers) and therefore we do not stop when the right place of the value $f_{n+1}(s)$ of the ``new'' function $f_{n+1}$ among the already ordered values $f_{n:1}(s),\dots,f_{n:n}(s)$ at a particular point $s\in S$ has been found, because this place will in general depend on $s$. So, the proof that \eqref{eq:g_k} implies \eqref{eq:ident} may be considered as (something a bit more than) a rigorous proof of the validity of the insertion search algorithm, avoiding such informal, undefined terms as swap, moving, and interleaving. 

\begin{proof}[Proof of Theorem~\ref{th:}]
Let us prove the theorem by induction in $n$. For $n=1$, the result is trivial. To make the induction step, it suffices to prove the following: For any natural $n\ge2$, if the function $\La\colon L^n\to\RR$ is generalized $(n\col2)$-submodular and the function $L^{n-1}\ni(f_{1},\dots,f_{n-1})\mapsto\La(f_1,\dots,f_n)$ is generalized $(n-1)$-submodular for each $f_n\in L$, then $\La$ is generalized $n$-sub\-modular. 
Thus, we are assuming that the function $\La\colon L^n\to\RR$ is generalized $(n\col2)$-submodular and 
\begin{equation}\label{eq:induct}
	\La(f_1,\dots,f_n)\l\La(f_{n-1:1},\dots,f_{n-1:n-1},f_n)
\end{equation}
for all $(f_1,\dots,f_n)\in L^n$, where $f_{n-1:1},\dots,f_{n-1:n-1}$ are the ``order statistics'' based on $(f_1,\dots,f_{n-1})$. 

Take indeed any $(f_1,\dots,f_n)\in L^n$. 
Define the rectangular array of functions $(g_{k,j}\colon k\in\intr0{n-1}, j\in[n])$ recursively, as follows: 
\begin{equation}\label{eq:g_0}
	(g_{0,1},\dots,g_{0,n-1},g_{0,n}):=(f_{n-1:1},\dots,f_{n-1:n-1},f_n) 
\end{equation}
and, for $k\in\intr1{n-1}$ and $j\in[n]$,  
\begin{equation}\label{eq:g_k}
	g_{k,j}:=
\left\{
\begin{alignedat}{2}
& g_{k-1,j} && \text{\ \ if\ \  }j\in\intr1{n-k-1}\,\cup\,\intr{n-k+2}n, \\ 
& g_{k-1,n-k}\wedge g_{k-1,n-k+1} && \text{\ \ if\ \  }j=n-k, \\ 
& g_{k-1,n-k}\vee g_{k-1,n-k+1} && \text{\ \ if\ \  }j=n-k+1.   
\end{alignedat}
\right.	
\end{equation}
By the \eqref{eq:induct} 
and \eqref{eq:g_0}, 
\begin{equation}\label{eq:>(0)}
	\La(f_1,\dots,f_n)\l\La(g_{0,1},\dots,g_{0,n-1},g_{0,n}).   
\end{equation}
Moreover, for each $k\in\intr1{n-1}$,  
\begin{equation}\label{eq:(k-1)>(k)}
	\La(g_{k-1,1},\dots,g_{k-1,n})
	\l\La(g_{k,1},\dots,g_{k,n}), 
\end{equation}
since $\La$ is generalized $(n\col2)$-submodular.  

It follows from \eqref{eq:>(0)} and \eqref{eq:(k-1)>(k)} that 
\begin{equation*}
	\La(f_1,\dots,f_n)\l
	\La(g_{n-1,1},\dots,g_{n-1,n}).   
\end{equation*}
It remains to verify the identity  
\begin{equation}\label{eq:ident}
	(g_{n-1,1},\dots,g_{n-1,n})\overset{\text{(?)}}=(f_{n:1},\dots,f_{n:n}).   
\end{equation} 

In accordance with Remark~\ref{rem:pointwise}, we may and shall assume that the distributive lattice $L$ is a lattice of nonnegative real-valued functions on a set $S$, so that \eqref{eq:ord} holds 
for each $s\in S$. 

In the remainder of the proof, fix any $s\in S$. Then   
\begin{equation*}
	\{\{g_{0,1}(s),\dots,g_{0,n}(s)\}\}=\{\{f_1(s),\dots,f_{n}(s)\}\},   
\end{equation*}
by \eqref{eq:g_0} and the second part of \eqref{eq:ord} used with $n-1$ in place of $n$;   
also, for each $k\in\intr1{n-1}$, 
\begin{equation*}
	\{\{g_{k,1}(s),\dots,g_{k,n}(s)\}=\{\{g_{k-1,1}(s),\dots,g_{k-1,n}(s)\}\},  
\end{equation*}
by \eqref{eq:g_k}. So, 
\begin{equation*}
	\{\{g_{n-1,1}(s),\dots,g_{n-1,n}(s)\}\}=\{\{f_1(s),\dots,f_n(s)\}\}.  
\end{equation*}

Therefore, to complete the proof of \eqref{eq:ident} and thus that of Theorem~\ref{th:}, it remains to show that 
\begin{equation}\label{eq:g incr}
	g_{n-1,1}(s)\overset{\text{(?)}}\le\cdots\overset{\text{(?)}}\le g_{n-1,n}(s),     
\end{equation}
which will follow immediately from 

\begin{lemma}\label{lem:}
For each $k\in\intr1{n-1}$, 
the following assertion is true for all $s\in S$: 
\begin{equation*}\tag{$A_k$}
\begin{gathered}
	g_{k,j}(s)\le g_{k,j+1}(s)\text{ for all }j\in\intr1{n-k-2}\,\cup\,\intr{n-k}{n-1}; \\ 
	\text{also, $g_{k,n-k-1}(s)\le g_{k,n-k+1}(s)$ if $k\le n-2$.} 
\end{gathered}
\end{equation*}
\end{lemma}

Indeed, \eqref{eq:g incr} is the first clause in assertion $(A_k)$ with $k=n-1$. 
Thus, what finally remains to prove Theorem~\ref{th:} is to present the following. 

\begin{proof}[Proof of Lemma~\ref{lem:}]
For simplicity, let us be dropping $(s)$ -- thus writing $g_{k,j},f_n,\dots$ in place of $g_{k,j}(s),f_n(s),\dots$. We shall prove Lemma~\ref{lem:} by induction in $k\in\intr1{n-1}$. Assertion $(A_1)$ means that $g_{1,1}\le \cdots\le g_{1,n-2}$, $g_{1,n-1}\le g_{1,n}$, and $g_{1,n-2}\le g_{1,n}$ if $1\le n-2$. So, in view of \eqref{eq:g_k} and \eqref{eq:g_0},  $(A_1)$ can be rewritten as follows: $f_{n-1:1}\le \cdots\le f_{n-1:n-2}$, $f_{n-1:n-1}\wedge f_n \le f_{n-1:n-1}\vee f_n$, and $f_{n-1:n-2} \le f_{n-1:n-1}\vee f_n$; all these inequalities are obvious. So, $(A_1)$ holds. 

Take now any $k\in\intr2{n-1}$ and suppose that $(A_{k-1})$ holds. We need to show that then $(A_{k})$ holds. 

For all $j\in\intr1{n-k-2}\,\cup\,\intr{n-k+2}{n-1}$, we have $j+1\in\intr1{n-k-1}
\cup\,\intr{n-k+2}n$, whence, by \eqref{eq:g_k} and the first clause of $(A_{k-1})$,  
$g_{k,j}=g_{k-1,j}\le g_{k-1,j+1}=g_{k,j+1}$. So, 
\begin{equation}\label{eq:g<g,induct}
	g_{k,j}\le g_{k,j+1}\quad\text{for }j\in\intr1{n-k-2}\,\cup\,\intr{n-k+2}{n-1}. 
\end{equation}

If $j=n-k$\, then, by \eqref{eq:g_k}, $g_{k,j}=g_{k-1,n-k}\wedge g_{k-1,n-k+1}\le g_{k-1,n-k}\vee g_{k-1,n-k+1}=g_{k,j+1}$.

If $j=n-k+1$ then the condition $k\in\intr2{n-1}$ implies $j\le n-1$, and so,  
by \eqref{eq:g_k} and the second and first clauses of $(A_{k-1})$, $g_{k,j}=g_{k-1,n-k}\vee g_{k-1,n-k+1}
\le g_{k-1,n-k+2}=g_{k-1,j+1}=g_{k,j+1}$. 

Thus, in view of \eqref{eq:g<g,induct}, the first clause of $(A_{k})$ holds. 
Also, if $k\le n-2$ then, by \eqref{eq:g_k} and the first clause of $(A_{k-1})$, 
$g_{k,n-k-1}=g_{k-1,n-k-1}\le g_{k-1,n-k}\le g_{k-1,n-k}\vee g_{k-1,n-k+1}=g_{k,n-k+1}$, so that the second clause of $(A_{k})$ holds as well. 
This completes the proof of Lemma~\ref{lem:}. 
\end{proof}

Thus, Theorem~\ref{th:} is proved. 
\end{proof}

\begin{proof}[Proof of Proposition~\ref{prop:schur}]
Take any $(f_1,\dots,f_n)\in L^n$. Corollary~B.3 in \cite{marsh-ol09} states that $x\prec y$ iff $x$ is in the convex hull of the set of all points obtained by permuting the coordinates of the vector $y$. 
Also, since the function $\la$ is nondecreasing, we have $\la(f_1\vee f_2)\ge\la(f_1)\vee\la(f_2)$. 
For any real $a,b,c$ such that $c\ge a\vee b$, we have $(a,b)=(1-t)(a+b-c,c)+t(c,a+b-c)$ for  $t=\frac{c-b}{2c-a-b}\in[0,1]$ if $c>(a+b)/2$ and for any $t\in[0,1]$ otherwise (that is, if $a=b=c$). 
So, the point $(a,b)$ is a convex combination of points $(a+b-c,c)$ and $(c,a+b-c)$. Using this fact for $a=\la(f_1)$, $b=\la(f_2)$, $c=\la(f_1\vee f_2)$, we see that 
\begin{equation*}
(\la(f_1),\dots,\la(f_n))
\prec(\la(f_1)+\la(f_2)-\la(f_1\vee f_2),\la(f_1\vee f_2),\la(f_3),\dots,\la(f_n)). 
\end{equation*}
Also, $\la(f_1\we f_2)\le\la(f_1)+\la(f_2)-\la(f_1\vee f_2)$, by the submodularity of $\la$. 
Therefore and because $F$ is nondecreasing (in each of its $n$ arguments) and Schur-concave, we conclude that  
\begin{multline*}
	F(\la(f_1\we f_2),\la(f_1\vee f_2),\la(f_3),\dots,\la(f_n)) \\ 
	\le F(\la(f_1)+\la(f_2)-\la(f_1\vee f_2),\la(f_1\vee f_2),\la(f_3),\dots,\la(f_n)) \\ 
	\le F(\la(f_1),\dots,\la(f_n)). 
\end{multline*}
Quite similarly, 
\begin{multline*}
	F(\la(f_1),\dots,\la(f_{i-1}),\la(f_i\we f_{i+1}),\la(f_i\vee f_{i+1}),\la(f_{i+2}),\dots,\la(f_n)) \\ 
	\le F(\la(f_1),\dots,\la(f_n))
\end{multline*}
for all $i\in\intr1{n-1}$, so that the function $F$ is indeed generalized $(n\col2)$-submodular and hence, by Theorem~\ref{th:}, generalized $n$-submodular. 
\end{proof}

\begin{proof}[Proof of Proposition~\ref{prop:potential}]
In view of Theorem~\ref{th:}, it is enough to show that the function $\La=\La_{\vpi,\psi}$ is generalized $(n\col2)$-submodular. 
Without loss of generality (w.l.o.g.), we may and shall assume that the function $\vpi$ is nondecreasing, since $\La_{\vpi^-,\psi}=\La_{\vpi,\psi}$, where $\vpi^-(u):=\vpi(-u)$ for all real $u$. 
Also, w.l.o.g.\ $\psi(0)=0$ and hence $\Psi(0)=0$. 


Take any $f=(f_1,\dots,f_n)\in L^n$. Then, letting 
\begin{equation}\label{eq:tPsi}
	\tPsi(g):=\Psi(g)+\Psi(-g)
\end{equation}
for $g\in L$, one has 
\begin{equation}
	\La(f_1,f_2,f_3,\dots,f_n) 
	=\tPsi(f_1-f_2)+\sum_{j=3}^n\big(\tPsi(g_j)+\tPsi(h_j)\big)+R, \label{eq:La=} 
\end{equation}
{where $g_j:=f_1-f_j$, $h_j:=f_2-f_j$, and $R:=\sum_{3\le j<k\le n}^n\tPsi(f_j-f_k)$. 
Since $f_1\wedge f_2-f_1\vee f_2=-|f_1-f_2|$, one similarly has
} 
\begin{equation}\label{eq:Laa=}
\begin{aligned}
	 \La(f_1\wedge f_2,f_1\vee f_2,f_3,\dots,f_n) 
	=&
\tPsi(|f_1-f_2|) 
+\sum_{j=3}^n\big(\tPsi(g_j\wedge h_j)+\tPsi(g_j\vee h_j)\big)+R. 
\end{aligned}	
\end{equation}
Next, 
\begin{align}
	\tPsi(f_1-f_2)&=\psi\Big(\int_S\vpi\circ(f_1-f_2)\dd\mu\Big)
	+\psi\Big(\int_S\vpi\circ(f_2-f_1)\dd\mu\Big), \label{eq:Psi12} \\ 
	\tPsi(|f_1-f_2|)&=\psi\Big(\int_S\vpi\circ|f_1-f_2|\dd\mu\Big)
	+\psi\Big(\int_S\vpi\circ(-|f_2-f_1|)\dd\mu\Big), \label{eq:Psi|12|} 
\end{align}
$\vpi\circ(f_1-f_2)+\vpi\circ(f_2-f_1)=\vpi\circ|f_1-f_2|+\vpi\circ(-|f_1-f_2|)$ and hence 
\begin{equation}\label{eq:ints=ints}
\int_S\vpi\circ(f_1-f_2)\dd\mu+\int_S\vpi\circ(f_2-f_1)\dd\mu= 
\int_S\vpi\circ|f_1-f_2|\dd\mu+\int_S\vpi\circ(-|f_2-f_1|)\dd\mu.  	
\end{equation}
Also, since $\vpi$ is nondecreasing, $\vpi\circ(f_1-f_2)\,\vee\,\vpi\circ(f_2-f_1)\le\vpi\circ|f_1-f_2|$ and hence 
\begin{equation}\label{eq:ints<int}
\int_S\vpi\circ(f_1-f_2)\dd\mu\;\vee\;\int_S\vpi\circ(f_2-f_1)\dd\mu\le 
\int_S\vpi\circ|f_1-f_2|\dd\mu. 
\end{equation} 
Since the function $\psi$ is convex, it follows from \eqref{eq:Psi12}, \eqref{eq:Psi|12|}, \eqref{eq:ints=ints}, and \eqref{eq:ints<int} that 
\begin{equation}\label{eq:12}
	\tPsi(f_1-f_2)\le\tPsi(|f_1-f_2|). 
\end{equation}

Further, take any $j\in\intr3n$. Then 
$\vpi\circ g_j+\vpi\circ h_j
=\vpi\circ(g_j\wedge h_j)+\vpi\circ(g_j\vee h_j)$. 
So, 
\begin{equation*}
\int_S(\vpi\circ g_j)\dd\mu+\int_S(\vpi\circ h_j)\dd\mu 
=\int_S\vpi\circ(g_j\wedge h_j)\dd\mu+\int_S\vpi\circ(g_j\vee h_j)\dd\mu. 	
\end{equation*}
Moreover, since $\vpi$ is nondecreasing, $\int_S\vpi\circ(g_j\vee h_j)\dd\mu$ is no less than each of the integrals 
$\int_S(\vpi\circ g_j)\dd\mu$ and $\int_S(\vpi\circ h_j)\dd\mu$. So, in view of \eqref{eq:Psi} and the convexity of the function $\psi$, one has 
$\Psi(g_j)+\Psi(h_j)\le\Psi(g_j\wedge h_j)+\Psi(g_j\vee h_j)$. 
Similarly, because $\int_S\vpi\circ(-(g_j\wedge h_j))\dd\mu$ is no less than each of the integrals 
$\int_S\vpi\circ(-g_j)\dd\mu$ and $\int_S\vpi\circ(-h_j)\dd\mu$, one has $\Psi(-g_j)+\Psi(-h_j)\le\Psi(-(g_j\wedge h_j))+\Psi(-(g_j\vee h_j))$. 
So, by \eqref{eq:tPsi}, 
$\tPsi(g_j)+\tPsi(h_j)\le\tPsi(g_j\wedge h_j)+\tPsi(g_j\vee h_j)$.  

Therefore, by \eqref{eq:La=},  \eqref{eq:Laa=}, and \eqref{eq:12},  
$\La(f_1,f_2,f_3,\dots,f_n)\le\La(f_1\wedge f_2,\break 
f_1\vee f_2,f_3,\dots,f_n)$. 
Similarly, $\La(f_1,\dots,f_{j-1},f_j,f_{j+1},f_{j+2},\dots,f_n)\le
\La(f_1,\dots,f_{j-1},\break 
f_j\wedge f_{j+1},f_j\vee f_{j+1},f_{j+2},\dots,f_n)$ for all $j\in\intr1{n-1}$. 

Thus, the function $\La$ is generalized $(n\col2)$-supermodular, and so, by Theorem~\ref{th:}, it is generalized $n$-supermodular. 
\end{proof}

\begin{proof}[Proof of Proposition~\ref{prop:multiadd}]
%
Fix any $(f_1,\dots,f_n)\in L^n$. Then, in view of the permutation symmetry of $\m$ defined by \eqref{eq:m bar},  
\begin{equation}\label{eq:La_m,f,g}
		\frac1{k!}\,\La_m(f_1,\dots,f_n)=\la_2(f_{n-1},f_n)+\la_1(f_{n-1})+\la_1(f_n)+\la_0,
\end{equation}	
where 
\begin{align*}
	\la_2(f,g)&:=\sum_{1\le i_1<\dots<i_{k-2}\le n-2}\m(f_{i_1},\dots,f_{i_{k-2}},f,g), \\ 
	\la_1(f)&:=\sum_{1\le i_1<\dots<i_{k-1}\le n-2}\m(f_{i_1},\dots,f_{i_{k-1}},f), \\ 
	\la_0&:=\sum_{1\le i_1<\dots<i_k\le n-2}\m(f_{i_1},\dots,f_{i_k}), \\ 
\end{align*}
Similarly, 
\begin{multline}\label{eq:La_m,we,vee}		 
		\frac1{k!}\,\La_m(f_1,\dots,f_{n-2},f_{n-1}\we f_n,f_{n-1}\vee f_n)	
		=\la_2(f_{n-1}\we f_n,f_{n-1}\vee f_n) \\ 
		+\la_1(f_{n-1}\we f_n)+\la_1(f_{n-1}\vee f_n)+\la_0.   
\end{multline}

Note that the function $\la_2\colon L^2\to\R$ is $2$-additive and permutation-symmetric, and the function $\la_1\colon L^2\to\R$ is additive. Take any $f$ and $g$ in $L$. Then $(f\vee g)\we f=f$ and 
$(f\vee g)\sm f=g\sm f$. So, by the additivity of $\la_1$ we have $\la_1(f\vee g)=\la_1(f)+\la_1(g\sm f)$, whereas $\la_1(f\we g)+\la_1(g\sm f)=\la_1(g)$. So, 
\begin{equation}\label{eq:la_1}
	\la_1(f\we g)+\la_1(f\vee g)
=\la_1(f\we g)+\la_1(f)+\la_1(g\sm f)
=	
	\la_1(f)+\la_1(g). 
\end{equation}
%
By the $2$-additivity and permutation symmetry of $\la_2$ and because the function $\la_2$ is $2$-additive, permutation-symmetric, and nonnegative, we have 
\begin{equation}\label{eq:la_2}
\begin{aligned}
	\la_2(f\we g,f\vee g)  
	&=\la_2(f\we g,f\sm g)+\la_2(f\we g,g) \\ 
	&=\la_2(f\we g,f\sm g)+\la_2(f,g)-\la_2(f\sm g,g) \\ 
	&=\la_2(f\we g,f\sm g)+\la_2(f,g)-\la_2(f\sm g,g\we f)-\la_2(f\sm g,g\sm f) \\ 
	&=\la_2(f,g)-\la_2(f\sm g,g\sm f) \\ 
	&\le\la_2(f,g). 
\end{aligned}	
\end{equation}
It follows from \eqref{eq:La_m,f,g}, \eqref{eq:La_m,we,vee}, \eqref{eq:la_1}, and \eqref{eq:la_2} (with $f=f_{n-1}$ and $g=f_n$) that 
\begin{equation*}
	\La_m(f_1,\dots,f_{n-2},f_{n-1}\we f_n,f_{n-1}\vee f_n)\le\La_m(f_1,\dots,f_n). 
\end{equation*}
Therefore, being permutation-symmetric, the function $\La_m$ is indeed generalized $(n\col 2)$-submodular. Hence, by Theorem~\ref{th:}, $\La_m$ is generalized $n$-submodular. 
\end{proof}

%

\bibliographystyle{abbrv}

\bibliography{P:/pCloudSync/mtu_pCloud_02-02-17/bib_files/
citations10.13.18a}

\def\cprime{$'$} \def\polhk#1{\setbox0=\hbox{#1}{\ooalign{\hidewidth
  \lower1.5ex\hbox{`}\hidewidth\crcr\unhbox0}}}
  \def\polhk#1{\setbox0=\hbox{#1}{\ooalign{\hidewidth
  \lower1.5ex\hbox{`}\hidewidth\crcr\unhbox0}}}
  \def\polhk#1{\setbox0=\hbox{#1}{\ooalign{\hidewidth
  \lower1.5ex\hbox{`}\hidewidth\crcr\unhbox0}}} \def\cprime{$'$}
  \def\polhk#1{\setbox0=\hbox{#1}{\ooalign{\hidewidth
  \lower1.5ex\hbox{`}\hidewidth\crcr\unhbox0}}} \def\cprime{$'$}
  \def\polhk#1{\setbox0=\hbox{#1}{\ooalign{\hidewidth
  \lower1.5ex\hbox{`}\hidewidth\crcr\unhbox0}}} \def\cprime{$'$}
  \def\cprime{$'$}
\begin{thebibliography}{10}

\bibitem{aharoni-keich}
R.~Aharoni and U.~Keich.
\newblock A generalization of the {A}hlswede-{D}aykin inequality.
\newblock {\em Discrete Math.}, 152(1-3):1--12, 1996.

\bibitem{ahls-daykin79}
R.~Ahlswede and D.~E. Daykin.
\newblock Inequalities for a pair of maps {$S\times S\rightarrow S$} with {$S$}
  a finite set.
\newblock {\em Math. Z.}, 165(3):267--289, 1979.

\bibitem{bach13}
F.~Bach.
\newblock {\em Learning with Submodular Functions: A Convex Optimization
  Perspective}, volume~6.
\newblock NOW, 2013.

\bibitem{bach15}
F.~R. Bach.
\newblock Submodular functions: from discrete to continous domains.
\newblock {\em CoRR}, abs/1511.00394, 2015.

\bibitem{borell73_pacific}
C.~Borell.
\newblock A note on an inequality for rearrangements.
\newblock {\em Pacific J. Math.}, 47:39--41, 1973.

\bibitem{choquet}
G.~Choquet.
\newblock Theory of capacities.
\newblock {\em Ann. Inst. Fourier, Grenoble}, 5:131--295 (1955), 1953--1954.

\bibitem{david-nagaraja}
H.~A. David and H.~N. Nagaraja.
\newblock {\em Order statistics}.
\newblock Wiley Series in Probability and Statistics. Wiley-Interscience [John
  Wiley \& Sons], Hoboken, NJ, third edition, 2003.

\bibitem{fkg}
C.~M. Fortuin, P.~W. Kasteleyn, and J.~Ginibre.
\newblock Correlation inequalities on some partially ordered sets.
\newblock {\em Comm. Math. Phys.}, 22:89--103, 1971.

\bibitem{fujishige}
S.~Fujishige.
\newblock {\em Submodular functions and optimization}, volume~47 of {\em Annals
  of Discrete Mathematics}.
\newblock North-Holland Publishing Co., Amsterdam, 1991.

\bibitem{HLP}
J.~E.~L. G.~H.~Hardy and G.~P\'olya.
\newblock {\em Inequalities}.
\newblock Cambridge, 1934.

\bibitem{graetzer11}
G.~Gr{\"a}tzer.
\newblock {\em Lattice theory: {F}oundation}.
\newblock Birkh\"auser/Springer Basel AG, Basel, 2011.

\bibitem{joag-dev--proschan}
K.~Joag-Dev and F.~Proschan.
\newblock Negative association of random variables, with applications.
\newblock {\em Ann. Statist.}, 11(1):286--295, 1983.

\bibitem{karlin-rinott80-I}
S.~Karlin and Y.~Rinott.
\newblock Classes of orderings of measures and related correlation
  inequalities. {I}. {M}ultivariate totally positive distributions.
\newblock {\em J. Multivariate Anal.}, 10(4):467--498, 1980.

\bibitem{karlin-rinott80-II}
S.~Karlin and Y.~Rinott.
\newblock Classes of orderings of measures and related correlation
  inequalities. {II}. {M}ultivariate reverse rule distributions.
\newblock {\em J. Multivariate Anal.}, 10(4):499--516, 1980.

\bibitem{knuth-vol3}
D.~E. Knuth.
\newblock {\em The art of computer programming. {V}ol. 3}.
\newblock Addison-Wesley, Reading, MA, 1998.
\newblock Sorting and searching, Second edition [of MR0445948].

\bibitem{lorentz53}
G.~G. Lorentz.
\newblock An inequality for rearrangements.
\newblock {\em Amer. Math. Monthly}, 60:176--179, 1953.

\bibitem{marsh-ol09}
A.~W. Marshall, I.~Olkin, and B.~C. Arnold.
\newblock {\em Inequalities: theory of majorization and its applications}.
\newblock Springer Series in Statistics. Springer, New York, second edition,
  2009.

\bibitem{milgrom-roberts}
P.~Milgrom and J.~Roberts.
\newblock Rationalizability, learning, and equilibrium in games with strategic
  complementarities.
\newblock {\em Econometrica}, 58(6):1255--1277, 1990.

\bibitem{narayanan}
H.~Narayanan.
\newblock {\em Submodular functions and electrical networks}, volume~54 of {\em
  Annals of Discrete Mathematics}.
\newblock North-Holland Publishing Co., Amsterdam, 1997.

\bibitem{left_publ}
I.~Pinelis.
\newblock Optimal binomial, {P}oisson, and normal left-tail domination for sums
  of nonnegative random variables.
\newblock {\em Electron. J. Probab.}, 21:1--19, 2016.

\bibitem{rinott-saks92}
Y.~Rinott and M.~Saks.
\newblock On {FKG}-type and permanental inequalities.
\newblock In {\em Stochastic inequalities ({S}eattle, {WA}, 1991)}, volume~22
  of {\em IMS Lecture Notes Monogr. Ser.}, pages 332--342. Inst. Math.
  Statist., Hayward, CA, 1992.

\bibitem{rinott-saks93}
Y.~Rinott and M.~Saks.
\newblock Correlation inequalities and a conjecture for permanents.
\newblock {\em Combinatorica}, 13(3):269--277, 1993.

\bibitem{ruderman}
H.~D. Ruderman.
\newblock Two new inequalities.
\newblock {\em Amer. Math. Monthly}, 59:29--32, 1952.

\bibitem{topkis}
D.~M. Topkis.
\newblock Equilibrium points in nonzero-sum {$n$}-person submodular games.
\newblock {\em SIAM J. Control Optim.}, 17(6):773--787, 1979.

\bibitem{topkis98}
D.~M. Topkis.
\newblock {\em Supermodularity and complementarity}.
\newblock Frontiers of Economic Research. Princeton University Press,
  Princeton, NJ, 1998.

\end{thebibliography}

\end{document}